\documentclass[12pt,a4paper]{article}

\usepackage{latexsym}
\usepackage{amsmath}
\usepackage{amssymb}
\usepackage{amsthm}
\usepackage{amscd}
\usepackage{mathrsfs}
\usepackage[all]{xy}
\usepackage{graphicx}
\usepackage{comment}
\usepackage{arydshln}
\usepackage{mathtools}

\mathtoolsset{showonlyrefs=true}

\newtheorem{thm}{Theorem}[section]
\newtheorem{prop}[thm]{Proposition}
\newtheorem{defn}[thm]{Definition}
\newtheorem{lem}[thm]{Lemma}
\newtheorem{cor}[thm]{Corollary}
\newtheorem{conj}[thm]{Conjecture}
\newtheorem{rem}[thm]{Remark}
\newtheorem{eg}[thm]{Example}

\makeatletter

\@addtoreset{equation}{section}
\makeatother

\makeatletter
\newcommand{\subsubsubsection}{\@startsection{paragraph}{4}{\z@}%
 {1.0\Cvs \@plus.5\Cdp \@minus.2\Cdp}%
 {.1\Cvs \@plus.3\Cdp}%
 {\reset@font\sffamily\normalsize}
 }
\makeatother
\newcommand{\relmiddle}[1]{\mathrel{}\middle#1\mathrel{}}

\DeclareMathOperator{\Gal}{Gal}

\DeclareMathOperator{\End}{End}
\DeclareMathOperator{\Hom}{Hom}
\DeclareMathOperator{\Ext}{Ext}
\DeclareMathOperator{\Rep}{Rep}

\DeclareMathOperator{\Aut}{Aut}

\DeclareMathOperator{\Irr}{Irr}

\DeclareMathOperator{\Image}{Im}

\DeclareMathOperator{\Lie}{Lie}
\DeclareMathOperator{\Art}{Art}

\DeclareMathOperator{\Ad}{Ad}
\DeclareMathOperator{\ad}{ad}

\newcommand{\bC}{\mathbb{C}}

\newcommand{\bG}{\mathbb{G}}

\newcommand{\bQ}{\mathbb{Q}}
\newcommand{\bR}{\mathbb{R}}

\newcommand{\bZ}{\mathbb{Z}}

\newcommand{\cC}{\mathcal{C}}

\newcommand{\cF}{\mathcal{F}}
\newcommand{\cG}{\mathcal{G}}
\newcommand{\cH}{\mathcal{H}}

\newcommand{\cL}{\mathcal{L}}
\newcommand{\cM}{\mathcal{M}}

\newcommand{\cT}{\mathcal{T}}

\newcommand{\sM}{\mathscr{M}}

\newcommand{\GL}{\mathrm{GL}}
\newcommand{\SL}{\mathrm{SL}}
\newcommand{\PGL}{\mathrm{PGL}}
\newcommand{\SU}{\mathrm{SU}}

\newcommand{\ol}{\overline}

\newcommand{\wh}{\widehat}
\newcommand{\wt}{\widetilde}
\newcommand{\lra}{\longrightarrow}

\newcommand{\cf}{\textit{cf.\ }}

\begin{document}

\title
{Local Langlands correspondences\\ in $\ell$-adic coefficients} 

\author{Naoki Imai}
\date{}
\maketitle
\begin{abstract}
Let $\ell$ be a prime number different from 
the residue characteristic of a non-archimedean local field $F$. 
We give formulations of $\ell$-adic local Langlands correspondences 
for connected reductive algebraic groups over $F$, 
which we conjecture to be independent of 
a choice of an isomorphism between the $\ell$-adic 
coefficient field and the complex number field. 
\end{abstract}


\section*{Introduction}

The local Langlands correspondence for a connected reductive algebraic group 
over a non-archimedean local field $F$ 
is usually formulated with 
coefficients in $\bC$ 
because of its relation with 
automorphic representations. 
On the other hand, 
when we discuss a realization of 
the local Langlands correspondence in $\ell$-adic cohomology, 
we need a correspondence over $\ol{\bQ}_{\ell}$, 
where $\ell$ is a prime number different from 
the residue characteristic of $F$. 
We can take an isomorphism 
$\iota \colon \bC \stackrel{\sim}{\to} \ol{\bQ}_{\ell}$ 
and use it to transfer the local Langlands correspondence 
over $\bC$ to 
a local Langlands correspondence over $\ol{\bQ}_{\ell}$. 
However, the obtained correspondence over $\ol{\bQ}_{\ell}$ 
depends on the choice of $\iota$. 
In \cite{BHLLCGL2}, Bushnell--Henniart formulated 
an $\ell$-adic local Langlands correspondence for $\GL_2$, 
which is independent of a choice of 
an isomorphism $\bC \stackrel{\sim}{\to} \ol{\bQ}_{\ell}$ 
by making some twists of L-parameters. 
The $\ell$-adic local Langlands correspondence is 
suitable to describe the non-abelian Lubin--Tate theory 
in the sense that it is canonically defined over $\ol{\bQ}_{\ell}$ (\cf \cite[5]{ITreal3}). 

In this paper, we discuss formulations of 
$\ell$-adic local Langlands correspondences for general connected 
reductive groups. 
A natural idea is to make similar twists as $\GL_2$-case 
for L-parameters. 
However, we can not make such twists in general as explained in 
Example \ref{eg:PGL}. 
A problem is that we do not have enough space 
inside the Langlands dual group to make an appropriate twist. 
To overcome this problem, we have two approaches. 
One is to introduce 
$\ell$-adic C-parameters using 
C-groups, which make spaces for twists. 
Another is to introduce 
Tannakian $\ell$-adic L-parameters that incorporate 
necessary twists in the realizations. 
Using these parameters, we formulate 
$\ell$-adic local Langlands correspondences, 
which we conjecture to be independent of 
a choice of an isomorphism $\bC \stackrel{\sim}{\to} \ol{\bQ}_{\ell}$. 
We show that two conjectures formulated by using 
$\ell$-adic C-parameters and Tannakian $\ell$-adic L-parameters 
are equivalent. 
Further, we confirm that 
the conjectures are true for $\GL_n$ and $\PGL_n$. 
The formulation of the 
$\ell$-adic local Langlands correspondence 
using Tannakian $\ell$-adic L-parameters is 
motivated by the Kottwitz conjecture for local Shimura varieties 
in \cite[Conjecture 7.4]{RVlocSh}. 
In the number field case, 
a relation between C-groups and the Kottwitz conjecture for Shimura varieties 
is discussed in \cite{JohremBG}. 

In Section \ref{sec:Lp}, we recall various versions of 
L-parameters and explain their relations. 
In Section \ref{sec:LLC}, after explaining the problem, 
we 
give formulations of the $\ell$-adic local Langlands correspondences 
introducing the $\ell$-adic C-parameter and 
the Tannakian $\ell$-adic L-parameter. 

After we put a former version of this paper on arXiv, related papers 
\cite{BerHidL} and \cite{ZhuIntSat} appeared on arXiv. 
\cite{BerHidL} also explains a relation between C-groups and 
the local Langlands correspondence. 
In a similar philosophy as this paper, 
formulations of Satake isomorphisms using C-groups 
are explained in \cite{BerHidL} and \cite{ZhuIntSat}. 

\subsection*{Acknowledgments}
The author would like to thank 
Teruhisa Koshikawa 
for his helpful comments. 
He is grateful to referees and an editor for their comments and fruitful suggestions. 
This work was supported by JSPS KAKENHI Grant Numbers JP18H01109, JP22H00093.  

\subsection*{Notation}
For a field $F$, let $\Gamma_F$ 
denote the absolute Galois group of $F$. 
For a non-archimedean local field $F$, let 
$W_F$ and $I_F$ denote the Weil group of $F$ and 
its inertia subgroup. 
For a homomorphism $\phi \colon G \to H$ of groups 
and $g \in G$, 
let $\phi^g$ denote the homomorphism 
defined by 
$\phi^g (g')=\phi (gg'g^{-1})$ for $g' \in G$. 
For a group $G$, let $Z(G)$ denote the center of $G$. 
For a group $G$, a subgroup $H \subset G$ and a subset $S \subset G$, 
let $Z_H(S)$ denote the centralizer of $S$ in $H$. 
Let $\Ad$ denote the adjoint action of a group, and 
$\ad$ denote the adjoint action of a Lie algebra. 

\section{Langlands parameters}\label{sec:Lp}

\subsection{Over fields of characteristic zero}

Let $F$ be a non-archimedean local field of residue characteristic $p$. 
Let $q$ be the number of elements of the residue field of $F$. 
Let $v_F \colon F^{\times} \to \bZ$ be the normalized valuation of $F$. 
Let 
\[
 \Art_F \colon F^{\times} \stackrel{\sim}{\lra} 
 W_F^{\mathrm{ab}} 
\]
be the Artin reciprocity isomorphism normalized so that 
a uniformizer is sent to 
a lift of the geometric Frobenius element. 
For $w \in W_F$, we put 
\[
 d_F(w)=v_F (\Art_F^{-1}(\ol{w})), \quad 
 \lvert w \rvert = q^{-d_F(w)} , 
\]
where $\ol{w}$ denotes the image of $w$ in $W_F^{\mathrm{ab}}$. 

Let $G$ be a connected reductive algebraic group over $F$. 
Let $C$ be a field of characteristic $0$. 
Let $\wh{G}$ be the dual group of $G$ over $C$. 
Let ${}^L G =\wh{G} \rtimes W_F$ 
be the L-group of $G$ defined in \cite[2.3]{BorAutL}. 
We say that an element of ${}^L G (C)$ is semisimple 
if its image in 
$\wh{G}(C) \rtimes \Gal(F'/F)$ is semisimple for any finite Galois extension 
$F'$ over $F$ that splits $G$. 
Let $p_{\wh{G}} \colon {}^L G(C) \to \wh{G}(C)$ 
denote the projection map. 

\begin{defn}
Let $H$ be a group with action of $W_F$. 
Let $\phi \colon H \rtimes W_F \to {}^L G(C)$ 
be a homomorphism of groups over $W_F$. 
\begin{enumerate}
\item 
We say that $\phi$ is semisimple 
if all the elements of $\phi (W_F)$ are semisimple. 
\item 
We say that $\phi$ is relevant if any parabolic 
subgroup $P$ of ${}^L G$ containing $\Image \phi$ is relevant in the sense of 
\cite[3.3]{BorAutL}. 
\end{enumerate}
We say that 
two homomorphisms 
$H \rtimes W_F \to {}^L G(C)$ 
of groups over $W_F$ are equivalent 
if they are conjugate by an element of $\wh{G}(C)$. 
\end{defn}

We say that a map from $W_F$ to a set is smooth if it is locally constant. 

\begin{lem}\label{lem:liftss}
Let $\phi \colon W_F \to {}^L G(C)$ 
be a homomorphism of groups over $W_F$ such that 
$p_{\wh{G}} \circ \phi$ is smooth. 
Let $\sigma_q \in W_F$ be a lift of 
the $q$-th power Frobenius element. 
Assume that $\phi(\sigma_q)$ is semisimple. 
Then $\phi$ is semisimple. 
\end{lem}
\begin{proof}
We write $w \in W_F$ as $\sigma_q^m \sigma$ for 
$m \in \bZ$ and $\sigma \in I_F$. 
Since $(p_{\wh{G}} \circ \phi)(I_F)$ is finite, 
there is a positive integer $d$ such that 
$\phi ((\sigma_q^m \sigma)^d)=\phi(\sigma_q^{dm})$. 
Then the claim follows, because 
$\phi(w)$ is semisimple if and only if 
$\phi(w)^d$ is semisimple. 
\end{proof}

Let $\mathit{WD}_F=\bG_{\mathrm{a}} \rtimes W_F$ 
be the Weil--Deligne group scheme over $\bQ$ for $F$ 
defined in \cite[8.3.6]{DelconstL}. 

\begin{defn}[\cf {\cite[8.2]{BorAutL}}]
\begin{enumerate}
\item 
An L-homomorphism of Weil--Deligne type for $G$ over $C$ 
is 
a homomorphism 
\[
 \xi \colon \mathit{WD}_F(C) \to {}^L G(C)
\] 
of groups 
over $W_F$ such that 
$\xi|_{\bG_{\mathrm{a}}(C)}$ is algebraic and 
$(p_{\wh{G}} \circ \xi)|_{W_F}$ is smooth. 
\item 
An L-parameter of Weil--Deligne type for $G$ over $C$ is 
a semisimple relevant L-homomorphism 
of Weil--Deligne type for $G$ over $C$. 
\end{enumerate}
\end{defn}

Let $\cL_C^{\mathrm{WD}}(G)$ denote the set of equivalence classes of 
L-homomorphisms of Weil--Deligne type for $G$ over $C$. 
Let $\Phi_C^{\mathrm{WD}}(G)$ denote the set of 
equivalence classes of 
L-parameters of Weil--Deligne type for $G$ over $C$.

\begin{defn}
\begin{enumerate}
\item 
A Weil--Deligne L-homomorphism for $G$ over $C$ is a pair $(\tau,N)$ of 
a homomorphism $\tau \colon W_F \to {}^L G(C)$ of groups over $W_F$ 
and $N \in \Lie (\wh{G})(C)$ such that 
$p_{\wh{G}} \circ \tau$ is smooth and 
$\Ad (\tau(w)) N= \lvert w \rvert N$ for $w \in W_F$. 
The second component $N$ of 
a Weil--Deligne L-homomorphism $(\tau,N)$ 
is called a monodromy operator. 
We say that two 
Weil--Deligne L-homomorphisms for $G$ over $C$ 
are equivalent if 
they are conjugate by an element of $\wh{G}(C)$.
\item 
A Weil--Deligne L-parameter for $G$ over $C$ is 
a Weil--Deligne L-homomorphism $(\tau,N)$ for $G$ over $C$ 
such that 
$\tau$ is semisimple and 
any parabolic 
subgroup $P$ of ${}^L G$ containing $\tau (W_F)$ 
and satisfying $N \in \Lie (P \cap \wh{G})(C)$ 
is relevant. 
\end{enumerate}
\end{defn}

Let $\cL_C^{\mathrm{M}}(G)$ denote the set of equivalence classes of 
Weil--Deligne L-homomorphisms for $G$ over $C$. 
Let $\Phi_C^{\mathrm{M}}(G)$ denote the set of 
equivalence classes of 
Weil--Deligne L-parameters for $G$ over $C$. 

\begin{rem}\label{rem:iotaind}
Let $\iota \colon C \stackrel{\sim}{\to} C'$ be an isomorphism of 
fields of characteristic $0$. 
Then $\iota$ induces bijections 
$\cL_C^{\mathrm{WD}}(G) \simeq \cL_{C'}^{\mathrm{WD}}(G)$, 
$\Phi_C^{\mathrm{WD}}(G) \simeq \Phi_{C'}^{\mathrm{WD}}(G)$, 
$\cL_C^{\mathrm{M}}(G) \simeq \cL_{C'}^{\mathrm{M}}(G)$
and 
$\Phi_C^{\mathrm{M}}(G) \simeq \Phi_{C'}^{\mathrm{M}}(G)$. 
\end{rem}

\begin{lem}\label{lem:Nnilp}
Let $(\tau,N)$ be a Weil--Deligne L-homomorphism for $G$ over $C$. 
Then $N$ is an nilpotent element of $\Lie (\wh{G}^{\mathrm{der}}(C))$. 
\end{lem}
\begin{proof}
We take a finite separable extension $F'$ of $F$ that splits $G$. 
By the condition 
$\Ad (\tau(w)) N= \lvert w \rvert N$ for $w \in W_{F'}$, 
we have $N \in \Lie (\wh{G}^{\mathrm{der}}(C))$. Further, 
the adjoint endomorphism 
$\ad (N) \in \End (\Lie (\wh{G}^{\mathrm{der}}(C)))$ satisfies 
\[
 \Ad (\tau (w)) \circ \ad (N) = \lvert w \rvert \ad (N) \circ \Ad (\tau (w)) 
\]
for $w \in W_F$. 
Hence $N$ is an nilpotent element. 
\end{proof}

\begin{prop}
For an L-homomorphism of Weil--Deligne type $\xi$ 
for $G$ over $C$, 
we put $\tau_{\xi}=\xi|_{W_F}$ and 
$N_{\xi}=\Lie(\xi|_{\bG_{\mathrm{a}}(C)})(1)$. 
Then $\xi \mapsto (\tau_{\xi},N_{\xi})$ 
induces bijections 
$\cL_C^{\mathrm{WD}}(G) \simeq \cL_C^{\mathrm{M}}(G)$ 
and 
$\Phi_C^{\mathrm{WD}}(G) \simeq \Phi_C^{\mathrm{M}}(G)$. 
\end{prop}
\begin{proof}
Let $(\tau,N)$ be a Weil--Deligne L-homomorphism for $G$ over $C$. 
By Lemma \ref{lem:Nnilp}, 
there is a unipotent radical $U$ of a Borel subgroup of $\wh{G}$ 
such that $N \in \Lie (U(C))$. 
Then there is an exponential map 
\[
 \exp \colon \Lie (U(C)) \to U(C) 
\]
as in \cite[2.2]{SerExPSL}. 
We define 
\[
 \xi_{(\tau,N)} \colon \mathit{WD}_F(C) \to {}^L G(C)
\]
by 
$\xi_{(\tau,N)}(a,w) =\exp (aN) \tau(w)$. 
Then $(\tau,N) \mapsto \xi_{(\tau,N)}$ 
defines the inverses. 
\end{proof}

\subsection{Over $\bC$} 
Assume that $C=\bC$ in this subsection. 

\begin{defn}[\cf {\cite[I.2]{ArGeLecL}}]
An L-parameter for $G$ is a semisimple relevant continuous homomorphism 
\[
 \phi \colon \SU_2(\bR) \times W_F \to {}^L G(\bC) 
\] 
of groups over $W_F$. 
\end{defn}

Let $\Phi(G)$ denote the set of 
equivalence classes of L-parameters for $G$. 

\begin{defn}[\cf {\cite[IV.2]{Landeb}}]
An L-parameter of $\SL_2$-type for $G$ is 
a semisimple relevant continuous homomorphism 
\[
 \phi \colon \SL_2(\bC) \times W_F \to {}^L G(\bC) 
\] 
of groups 
over $W_F$ such that 
$\phi|_{\SL_2(\bC)}$ is algebraic. 
\end{defn}

Let $\Phi^{\SL} (G)$ denote the set of 
equivalence classes of 
L-parameters of 
$\SL_2$-type for $G$. 

\begin{prop}
The map $\Phi^{\SL} (G) \to \Phi(G)$ 
induced by the restriction with respect to 
$\SU_2(\bR) \subset \SL_2(\bC)$ 
is a bijection. 
\end{prop}
\begin{proof}
Let $\phi \colon \SU_2(\bR) \times W_F \to {}^L G(\bC)$ 
be an L-parameter for $G$. 
Let $\cH$ be the centralizer of $\phi(W_F)$ in $\wh{G}$. 
Then $\cH$ is reductive by \cite[10.1.1 Lemma]{KotStfcus}. 
We take a compact real form $\cH_{\mathrm{c}}$ of $\cH$ 
such that 
$\phi(\SU_2(\bR)) \subset \cH_{\mathrm{c}}(\bR)$. 
Let $\phi_{\SU_2} \colon \SU_2(\bR) \to \cH_{\mathrm{c}}(\bR)$ 
be the restriction of $\phi$ to $\SU_2(\bR)$. 
The continuous homomorphism 
$\phi_{\SU_2}$ 
extends to an algebraic morphism 
$\phi_{\SL_2} \colon \SL_2(\bC) \to \cH(\bC)$ 
uniquely by \cite[5.2.5 Theorem 11]{OnViLieag}, 
since any continuous homomorphism between real Lie groups 
are differentiable. 
Let 
$\phi^{\SL} \colon \SL_2(\bC) \times W_F \to {}^L G(\bC)$ 
be a homomorphism defined by 
$\phi_{\SL_2}$ and $\phi|_{W_F}$. 
Then $\phi \mapsto \phi^{\SL}$ induces the inverse of 
the map $\Phi^{\SL} (G) \to \Phi(G)$. 
\end{proof}

\begin{lem}\label{lem:centred}
Let $H$ be a compact topological group. 
Let $\phi \colon H \to {}^L G(\bC)$ be a continuous homomorphism. 
Then the centralizer of $\phi (H)$ in $\wh{G}$ is reductive. 
\end{lem}
\begin{proof}
Let $K$ be the image of $\phi (H)$ under 
$\Ad \colon {}^L G \to \Aut(\wh{G})$. 
Then $\wh{G}^K$ is reductive 
as in the proof of \cite[10.1.2 Lemma]{KotStfcus}. 
Hence we obtain the claim. 
\end{proof}

The following is a slight generalization of 
\cite[Proposition 3.5]{HeiOrbpole} and \cite[2.4]{KaLuDLH}. 

\begin{lem}\label{lem:SLmor}
Let $\cG$ be a reductive group over $\bC$. 
Let $s$ be a semisimple element of $\cG(\bC)$ and 
$u$ be a unipotent element of $\cG(\bC)$ 
such that $sus^{-1}=u^q$. 
Then there is an algebraic homomorphism 
$\theta \colon \SL_2(\bC) \to \cG(\bC)$ 
such that 
\[
 \theta \left( 
 \begin{pmatrix}
 1 & 1 \\ 
 0 & 1
 \end{pmatrix}
 \right) =u , \quad 
 \theta \left( 
 \begin{pmatrix}
 q^{\frac{1}{2}} & 0 \\ 
 0 & q^{-\frac{1}{2}} 
 \end{pmatrix}
 \right) s^{-1} \in 
 Z_{\cG(\bC)}(\Image \theta ). 
\]
Further, such $\theta$ is unique up to conjugation by 
$Z_{\cG(\bC)}(\{s,u \})$. 
\end{lem}
\begin{proof}
The first claim is proved in the same way as 
\cite[Proposition 3.5]{HeiOrbpole}. 
We recall the argument briefly. 
We take 
an algebraic homomorphism 
$\theta \colon \SL_2(\bC) \to \cG(\bC)$ 
such that 
\[
 \theta \left( 
 \begin{pmatrix}
 1 & 1 \\ 
 0 & 1
 \end{pmatrix}
 \right) =u . 
\]
We put $\cG'=\cG \times \bG_{\mathrm{m}}$. We define 
\begin{align*}
 S_{\cG'(\bC)}(u) &=\{ (g,z) \in \cG'(\bC) \mid gug^{-1} = u^{z^2} \}, \\ 
 S_{\cG'(\bC)}( \theta) 
 &=\left\{ (g,z) \in \cG'(\bC) \relmiddle{|} \Ad (g) \circ \theta = \theta \circ \Ad \left( 
 \begin{pmatrix}
 z & 0 \\ 
 0 & z^{-1} 
 \end{pmatrix}
 \right)  \right\} . 
\end{align*}
Then we can see that 
$S_{\cG'(\bC)}(\theta)$ is a maximal reductive subgroup of 
$S_{\cG'(\bC)}(u)$ as in the proof of \cite[Proposition 2.4]{BaVoUnipss}. 
Then 
any maximal reductive subgroup of $S_{\cG'(\bC)}(u)$ 
is conjugate to 
$S_{\cG'(\bC)}(\theta)$ 
by \cite[VIII. Theorem 4.3]{HocBasalg} (\cf \cite[Theorem 7.1]{MosFull}). 
Hence, by replacing $\theta$ by its conjugate under 
$S_{\cG'(\bC)}(u)$, 
we may assume that 
$(s,q^{\frac{1}{2}}) \in S_{\cG'(\bC)}(\theta)$. 
Then $\theta$ satisfies the conditions in the first claim. 

The second claim is proved in the same way as 
\cite[2.4 (h)]{KaLuDLH}. 
\end{proof}

Let $\phi \colon \SL_2(\bC) \times W_F \to {}^L G (\bC)$ 
be an L-parameter of $\SL_2$-type 
for $G$. 
We define $\xi_{\phi} \colon \mathit{WD}_F(\bC) \to {}^L G (\bC)$ by 
\[
 \xi_{\phi} (a,w)=\phi \left( 
\begin{pmatrix}
 1 & a \\ 
 0 & 1 
 \end{pmatrix} 
 \begin{pmatrix}
 \lvert w\rvert^{\frac{1}{2}} & 0 \\ 
 0 & \lvert w\rvert^{-\frac{1}{2}} 
 \end{pmatrix}, w 
 \right) . 
\]
We define $\Xi \colon \Phi^{\SL} (G) \to \Phi_{\bC}^{\mathrm{WD}}(G)$ 
by $\Xi ([\phi])=[\xi_{\phi}]$. 

The following proposition is proved in \cite[Proposition 2.2]{GRAinv}. 
We give a different proof here. 

\begin{prop}\label{prop:SLWD}
The map $\Xi$ is a bijection. 
\end{prop}
\begin{proof}
Let $\xi$ be a 
Weil--Deligne L-parameter for $G$ over $\bC$. 
Let $\cH$ be the centralizer of $\xi (I_F)$ in $\wh{G}$. 
Then $\cH$ is reductive by Lemma \ref{lem:centred}. 
Take a lift $\sigma_q \in W_F$ of the $q$-th power Frobenius element. 
Then $\cH$ is stable under conjugation by 
$\xi (\sigma_q)$, since $I_F$ is a normal subgroup of $W_F$. 
Take a positive integer $m_0$ such that 
$\sigma_q^{m_0}$ commutes with $\wh{G}$ in ${}^L G$ and 
\[
 (p_{\wh{G}} \circ \xi) (0,\sigma_q^{m_0} \sigma \sigma_q^{-m_0})
 =(p_{\wh{G}} \circ \xi) (0,\sigma) 
\]
for any $\sigma \in I_F$. 
Then $(1,\sigma_q^{m_0})^{\bZ}$ is a normal subgroup of 
$\cH \cdot \xi (\sigma_q)^{\bZ} \subset {}^L G$. 
We put 
\[
 \cG = \left( \cH \cdot \xi (\sigma_q)^{\bZ} \right)/ 
 (1,\sigma_q^{m_0})^{\bZ}. 
\]
Then $\cG$ is a reductive group, because 
the identity component of $\cG$ is equal to 
the identity component of $\cH$. 
We view $\cH$ as an algebraic subgroup of $\cG$. 
Let $s$ be the image of $\xi (0,\sigma_q)$ in $\cG(\bC)$. 
Since $\xi$ is semisimple, the element $s$ is semisimple. 
Let $u \in \wh{G}(\bC)$ 
be the image of $1 \in \bG_{\mathrm{a}}(\bC)$ under $\xi$. 
Then $u$ belongs to $\cH(\bC)$. 
Hence we can view $u$ as an element of $\cG(\bC)$. 
By Lemma \ref{lem:SLmor}, 
there is 
a morphism $\theta \colon \SL_2 (\bC) \to \cG(\bC)$ 
such that 
\begin{equation}\label{eq:thecond}
 \theta \left( 
 \begin{pmatrix}
 1 & 1 \\ 
 0 & 1
 \end{pmatrix}
 \right) =u , \quad 
 \theta \left( 
 \begin{pmatrix}
 q^{\frac{1}{2}} & 0 \\ 
 0 & q^{-\frac{1}{2}} 
 \end{pmatrix}
 \right) s^{-1} \in 
 Z_{\cG(\bC)}(\Image \theta ). 
\end{equation}
Since $\theta$ factors through the identity component of 
$\cG(\bC)$, it factors through $\cH(\bC)$. 
Hence $\theta$ determines a morphism 
$\theta_{\wh{G}} \colon \SL_2(\bC) \to \wh{G}(\bC)$. 
We note that 
\[
 \theta \left( 
 \begin{pmatrix}
 q^{\frac{1}{2}} & 0 \\ 
 0 & q^{-\frac{1}{2}} 
 \end{pmatrix}
 \right) 
\]
commutes with $\xi (W_F)$, 
because it commutes with $s$ 
by the latter condition in \eqref{eq:thecond} 
and we have $\Image \theta_{\wh{G}} \subset \cH(\bC)$. 
We define 
$\phi_{\xi} \colon \SL_2(\bC) \times W_F \to {}^L G (\bC)$ 
by $\phi_{\xi}|_{\SL_2(\bC)}=\theta_{\wh{G}}$ and 
\[
 \phi_{\xi} (1,w) =\theta_{\wh{G}}\left( 
 \begin{pmatrix}
 \lvert w \rvert^{-\frac{1}{2}} & 0 \\ 
 0 & \lvert w \rvert^{\frac{1}{2}} 
 \end{pmatrix}
 \right) \xi (0,w) . 
\]
Let $\theta'$ be another choice of $\theta$. 
Then $\theta'=\Ad(g') \theta$ for some $g' \in Z_{\cG(\bC)}(\{s,u\})$ by 
Lemma \ref{lem:SLmor}. 
Since we have 
\[
 \theta \left( 
 \begin{pmatrix}
 q^{\frac{1}{2}} & 0 \\ 
 0 & q^{-\frac{1}{2}} 
 \end{pmatrix}
 \right) s^{-1} \in 
 Z_{\cG(\bC)}(\Image \theta ) \cap 
 Z_{\cG(\bC)}(\{s,u\})
\]
by \eqref{eq:thecond}, 
we may replace $g'$ and assume that 
$g' \in Z_{\cH(\bC)}(\{s,u\})$. 
Then $g'$ commutes with $\xi(W_F)$. 
Hence $[\phi_{\xi}]$ is independent of the choice of $\theta$. 
We can see that $\xi \mapsto [\phi_{\xi}]$ induces the inverse of 
$\Xi$. 
\end{proof}

\begin{rem}
The bijectivity in Proposition \ref{prop:SLWD} does not hold in general if we drop the Frobenius semisimplicity conditions from the both sides (\cf \cite[Example 3.5]{BMIYJMmor}). 
\end{rem}

\subsection{Over $\ol{\bQ}_{\ell}$} 
Let $\ell$ be a prime number different from $p$. 
Assume that $C=\ol{\bQ}_{\ell}$ in this subsection. 

\begin{defn}
\begin{enumerate}
\item 
An $\ell$-adic L-homomorphism for $G$ is 
a continuous homomorphism 
\[
 \varphi \colon W_F \to {}^L G(\ol{\bQ}_{\ell})
\] 
of groups over $W_F$. 
We say that an 
$\ell$-adic L-homomorphism 
$\varphi$ for $G$ is Frobenius-semisimple 
if $\varphi (\sigma_q)$ is semisimple for any 
lift $\sigma_q \in W_F$ of the $q$-th power Frobenius element. 
\item 
An $\ell$-adic L-parameter for $G$ is 
an Frobenius-semisimple 
relevant $\ell$-adic L-homomorphism for $G$. 
\end{enumerate}
\end{defn}

Let $\cL_{\ell}(G)$ denote the set of the equivalence classes of 
$\ell$-adic L-homomorphisms. 
Let $\Phi_{\ell}(G)$ denote the set of the equivalence classes 
of $\ell$-adic L-parameters of $G$. 

Let 
$t_{\ell} \colon I_F \to \bZ_{\ell}(1)$ 
be the $\ell$-adic tame character. 
We take an isomorphism 
$\bZ_{\ell}(1) \simeq \bZ_{\ell}$ 
and let 
\[
 t'_{\ell} \colon I_F \stackrel{t_{\ell}}{\lra} 
 \bZ_{\ell}(1) \simeq \bZ_{\ell} . 
\]
Let $\xi$ be an L-homomorphism 
of Weil--Deligne type for $G$ 
over $\ol{\bQ}_{\ell}$. 
We take a lift $\sigma_q \in W_F$ of the $q$-th power Frobenius element 
and define $\varphi_{\xi} \colon W_F \to {}^L G(\ol{\bQ}_{\ell})$ 
by 
\[
 \varphi_{\xi} (\sigma_q^m \sigma) 
 =\xi \left( t'_{\ell} (\sigma) , \sigma_q^m \sigma  \right) 
\]
for $m \in \bZ$ and $\sigma \in I_F$. 

\begin{lem}\label{lem:indepsiso}
The equivalence class $[\varphi_{\xi}] \in \cL_{\ell}(G)$ 
of $\varphi_{\xi}$ constructed above 
is independent of choices of $\sigma_q$ and an isomorphism 
$\bZ_{\ell}(1) \simeq \bZ_{\ell}$. 
\end{lem}
\begin{proof}
Let $\sigma_q'$ be another choice of a lift of the 
$q$-th power Frobenius element. 
Then $\sigma_q'=\sigma_q \sigma'$ for some $\sigma' \in I_F$. 
We define $\varphi_{\xi}'$ similarly as $\varphi_{\xi}$ 
using $\sigma_q'$ instead of $\sigma_q$. 
We put 
\[
 g=\xi \left( \frac{t_{\ell}'(\sigma')}{q-1},1 \right). 
\]
Then we have 
\[
 \Ad (g) (\varphi_{\xi}(\sigma)) =\xi(t_{\ell}'(\sigma),\sigma)
 =\varphi_{\xi}'(\sigma) 
\]
for $\sigma \in I_F$, and 
\[
 \Ad (g) (\varphi_{\xi}(\sigma_q)) 
 =g \left( \Ad (\varphi_{\xi}(\sigma_q))(g^{-1}) \right) 
 \varphi_{\xi}(\sigma_q) 
 =\xi (-t_{\ell}'(\sigma'),\sigma_q) 
 = \varphi_{\xi}'(\sigma_q). 
\]
Hence we have $[\varphi_{\xi}]=[\varphi_{\xi}']$. 

Let 
\[
 t''_{\ell} \colon I_F \stackrel{t_{\ell}}{\lra} 
 \bZ_{\ell}(1) \simeq \bZ_{\ell} 
\]
be a homomorphism obtained from another choice of 
an isomorphism $\bZ_{\ell}(1) \simeq \bZ_{\ell}$. 
Then we have 
$t_{\ell}''=u t_{\ell}'$ for some $u \in \bZ_{\ell}^{\times}$. 
Take a positive integer $m_0$ such that 
$\sigma_q^{m_0}$ commutes with $\wh{G}$ in ${}^L G$ and 
\begin{equation}\label{eq:m0cond}
 (p_{\wh{G}} \circ \xi) (0,\sigma_q^{m_0} \sigma \sigma_q^{-m_0})
 =(p_{\wh{G}} \circ \xi) (0,\sigma)
\end{equation}
for any $\sigma \in I_F$. 
We put $h_0=(p_{\wh{G}} \circ \xi) (0,\sigma_q^{m_0})$. 
If $\xi|_{\bG_{\mathrm{a}}(\ol{\bQ}_{\ell})}$ 
is trivial, there is nothing to prove. 
Hence we assume that $\xi|_{\bG_{\mathrm{a}}(\ol{\bQ}_{\ell})}$ 
is non-trivial. 
Let $U_{\xi}$ be the algebraic subgroup of $\wh{G}$ defined by 
$\xi (\bG_{\mathrm{a}}(\ol{\bQ}_{\ell}))$. 
Let $\cH$ be the intersection of 
the normalizer of $U_{\xi}$ in $\wh{G}$ 
and 
the centralizer of $\xi(W_F)$ in $\wh{G}$. 
We have 
\[ 
\Ad (h_0)(\xi(a,1))=\Ad (\xi (\sigma_q^{m_0}))(\xi(a,1))=\xi(q^{m_0} a,1) 
\]
for $a \in \bG_{\mathrm{a}}(\ol{\bQ}_{\ell})$ and 
\begin{align*}
 \Ad (h_0)(\xi(0,\sigma_q))
 &=\Ad ((1,\sigma_q^{-m_0})) \Ad (\xi(0,\sigma_q^{m_0})) 
 (\xi(0,\sigma_q)) = \xi(0,\sigma_q),\\ 
 \Ad (h_0)(\xi(0,\sigma))
 &=\Ad ((1,\sigma_q^{-m_0})) (\xi(0,\sigma_q^{m_0} \sigma \sigma_q^{-m_0})) 
 =\xi(0,\sigma)
\end{align*}
for $\sigma \in I_F$ using \eqref{eq:m0cond}. 
Hence we have $h_0 \in \cH(\ol{\bQ}_{\ell})$. 
The morphism 
\[
 f \colon \cH \to \Aut (U_{\xi}) 
 \simeq \bG_{\mathrm{m}} 
\]
induced by the adjoint action 
is surjective, because $f (h_0)=q^{m_0}$ is not of finite order. 
Hence we can take $h \in \cH(\ol{\bQ}_{\ell})$ 
such that $f (h)=u$. 
Then we have 
\[
 \Ad (h) \left( \xi \left( t'_{\ell} (\sigma) , \sigma_q^m \sigma 
 \right) \right) 
 = \left( t''_{\ell} (\sigma) , \sigma_q^m \sigma  \right) 
\]
for $m \in \bZ$ and $\sigma \in I_F$. 
Therefore 
$[\varphi_{\xi}]$ 
is independent of a choice an isomorphism 
$\bZ_{\ell}(1) \simeq \bZ_{\ell}$. 
\end{proof}

We define 
$\Theta \colon 
 \cL_{\ol{\bQ}_{\ell}}^{\mathrm{WD}}(G) \to \cL_{\ell}(G)$ 
by $\Theta ([\xi])=[\varphi_{\xi}]$. 

\begin{prop}\label{prop:PhiWDell}
The map $\Theta$ is a bijection. 
Further it induces a bijection 
$\Phi_{\ol{\bQ}_{\ell}}^{\mathrm{WD}}(G) \to \Phi_{\ell}(G)$. 
\end{prop}
\begin{proof}
Let $\varphi$ be an $\ell$-adic L-homomorphism for $G$. 
Take a finite Galois extension $F'$ of $F$ such that $G$ splits over $F'$. 
Take a representation 
\[
 \eta \colon {}^L G (\ol{\bQ}_{\ell}) \to \GL_n (\ol{\bQ}_{\ell}) 
\]
which factors through a faithful algebraic representation 
\[
 \ol{\eta} \colon \wh{G}(\ol{\bQ}_{\ell}) \rtimes \Gal (F'/F) 
 \to \GL_n (\ol{\bQ}_{\ell}). 
\]
Applying Grothendieck's monodromy theorem (\cf \cite[Appendix, Proposition]{SeTaGred}) to $\eta \circ \varphi$, 
we obtain a homomorphism 
\[
 \xi_{\GL_n} \colon \mathit{WD}_F(\ol{\bQ}_{\ell}) \to 
 \GL_n (\ol{\bQ}_{\ell})
\] 
such that 
$\xi_{\GL_n}|_{\bG_{\mathrm{a}}(\ol{\bQ}_{\ell})}$ is algebraic, 
$\xi_{\GL_n}|_{W_F}$ is smooth and 
\[
 \xi_{\GL_n} \left( t_{\ell}' (\sigma) , \sigma_q^m \sigma  \right) 
 = (\eta \circ \varphi) (\sigma_q^m \sigma) 
\]
for $m \in \bZ$ and $\sigma \in I_F$. 
Take a finite separable extension $F''$ of $F'$ such that 
$\xi_{\GL_n}|_{I_{F''}}$ is trivial. 
Since $t'_{\ell}(I_{F''})$ is Zariski dense in 
$\bG_{\mathrm{a}}(\ol{\bQ}_{\ell})$, 
the algebraic morphism 
$\xi_{\GL_n}|_{\bG_{\mathrm{a}}(\ol{\bQ}_{\ell})}$ 
factors through 
the inclusion  
\[
 \wh{G}(\ol{\bQ}_{\ell}) \hookrightarrow 
 \wh{G}(\ol{\bQ}_{\ell}) \rtimes \Gal (F'/F) 
 \stackrel{\ol{\eta}}{\hookrightarrow} 
 \GL_n (\ol{\bQ}_{\ell})
\]
via an algebraic morphism 
$\alpha \colon \bG_{\mathrm{a}}(\ol{\bQ}_{\ell}) \to \wh{G}(\ol{\bQ}_{\ell})$. 
We define a homomorphism 
\[
 \xi_{\varphi} \colon \mathit{WD}_F(\ol{\bQ}_{\ell}) \to 
 {}^L G (\ol{\bQ}_{\ell}) 
\]
by 
\[
 \xi_{\varphi} \left(a , \sigma_q^m \sigma  \right) 
 =\alpha (a - t'_{\ell}(\sigma)) \varphi (\sigma_q^m \sigma) 
\]
for $a \in \bG_{\mathrm{a}}(\ol{\bQ}_{\ell})$, 
$m \in \bZ$ and $\sigma \in I_F$. 
Then $\varphi \mapsto \xi_{\varphi}$ induces the inverse of $\Theta$. 
The bijection $\Theta$ 
induces a bijection 
$\Phi_{\ol{\bQ}_{\ell}}^{\mathrm{WD}}(G) \to \Phi_{\ell}(G)$ 
by Lemma \ref{lem:liftss} and Lemma \ref{lem:indepsiso}. 
\end{proof}

\begin{cor}
Let $\sigma_q \in W_F$ be a 
lift  of the $q$-th power Frobenius element. 
Then an 
$\ell$-adic L-homomorphism 
$\varphi$ for $G$ is Frobenius-semisimple 
if $\varphi (\sigma_q)$ is semisimple. 
\end{cor}
\begin{proof}
By Proposition \ref{prop:PhiWDell}, 
we take an L-homomorphism $\xi$ 
of Weil--Deligne type for $G$ 
over $\ol{\bQ}_{\ell}$ such that 
$[\varphi]=[\varphi_{\xi}]$, where 
$\varphi_{\xi}$ is defined using $\sigma_q$. 
Then $\xi|_{W_F}$ is semisimple by Lemma \ref{lem:liftss}, 
because $\xi(0,\sigma_q)=\varphi_{\xi}(\sigma_q)$ is semisimple. 
Let $\sigma_q' \in W_F$ be 
another lift of the $q$-th power Frobenius element. 
We define $\varphi_{\xi}'$ using $\sigma_q'$. 
Then $[\varphi]=[\varphi_{\xi}']$ by Lemma \ref{lem:indepsiso}. 
Hence $\varphi(\sigma_q')$ is semisimple, 
because $\varphi_{\xi}'(\sigma_q')=\xi(0,\sigma_q')$ is semisimple. 
\end{proof}

\section{Local Langlands correspondence}\label{sec:LLC}

\subsection{Problem}

Let $\Irr (G(F))$ denote the set of isomorphism classes of 
irreducible smooth representations of $G(F)$ over $\bC$. 
The conjectured local Langlands correspondence is a surjective map 
\[
 \mathrm{LL}_G \colon \Irr (G(F)) \to \Phi (G) 
\]
with finite fibers satisfying 
various properties 
(\cf \cite[10]{BorAutL}, \cite[Conjecture G]{KalLLCnqs}). 
We assume the existence of the 
local Langlands correspondence in the sequel. 

Let $\Irr_{\ell} (G(F))$ denote the set of isomorphism classes of 
irreducible smooth representations of $G(F)$ over $\ol{\bQ}_{\ell}$. 
If we fix an isomorphism $\iota \colon \bC \stackrel{\sim}{\to} \ol{\bQ}_{\ell}$, 
we have a surjection 
\[
 \mathrm{LL}_{G,\ell}^{\iota} \colon \Irr_{\ell} (G(F)) \to \Phi_{\ell} (G) 
\]
sending $[\pi] \in \Irr_{\ell} (G(F))$ to 
the image of 
$\mathrm{LL}_G ([\pi \otimes_{\ol{\bQ}_{\ell},\iota^{-1}} \bC])$ 
under the bijection 
\[
 \Phi (G) \simeq \Phi^{\SL} (G) \simeq \Phi_{\bC}^{\mathrm{WD}} (G) 
 \stackrel{\Phi_{\iota}^{\mathrm{WD}} (G)}{\simeq} 
 \Phi_{\ol{\bQ}_{\ell}}^{\mathrm{WD}} (G) 
 \simeq \Phi_{\ell} (G), 
\]
where $\Phi_{\iota}^{\mathrm{WD}} (G)$ is a bijection induced by $\iota$ 
as in Remark \ref{rem:iotaind}. 
However, $\mathrm{LL}_{G,\ell}^{\iota}$ depends on the choice of 
an isomorphism $\iota \colon \bC \stackrel{\sim}{\to} \ol{\bQ}_{\ell}$. 

In \cite[35.1]{BHLLCGL2}, an $\ell$-adic local Langlands correspondence 
for $\GL_2$ 
is constructed. We recall the construction here. 
For $[\phi] \in \Phi (\GL_2)$, 
we define $[\wt{\phi}] \in \Phi (\GL_2)$ 
by 
\[
 \wt{\phi}(g,w)=\left( 
 \begin{pmatrix}
 \lvert w \rvert^{\frac{1}{2}} & 0 \\ 
 0 & \lvert w \rvert^{\frac{1}{2}} 
 \end{pmatrix},
 1 \right) \phi(g,w) . 
\]
We define a bijection 
$\Lambda \colon \Phi (\GL_2) \to \Phi (\GL_2)$ 
by $[\phi] \mapsto [\wt{\phi}]$. 
We define 
\[
 \mathrm{LL}_{\GL_2,\ell} \colon 
 \Irr_{\ell} (\GL_2(F)) \to \Phi_{\ell} (\GL_2) 
\]
by sending $[\pi] \in \Irr_{\ell} (\GL_2(F))$ to 
the image of 
$\mathrm{LL}_{\GL_2}([\pi \otimes_{\ol{\bQ}_{\ell},\iota^{-1}} \bC])$ 
under the bijection 
\[
 \Phi (\GL_2) \stackrel{\Lambda}{\simeq} \Phi (\GL_2) \simeq 
 \Phi^{\SL} (\GL_2) \simeq 
 \Phi_{\bC}^{\mathrm{WD}} (\GL_2) 
 \stackrel{\Phi_{\iota}^{\mathrm{WD}} (\GL_2)}{\simeq} \Phi_{\ol{\bQ}_{\ell}}^{\mathrm{WD}} (\GL_2) 
 \simeq \Phi_{\ell} (\GL_2) 
\]
using an isomorphism 
$\iota \colon \bC \stackrel{\sim}{\to} \ol{\bQ}_{\ell}$. 
Then $\mathrm{LL}_{\GL_2,\ell}$ is independent of 
the choice of $\iota$. 

We can not make a similar twist for an L-parameter of 
a general connected reductive group $G$ 
as the following example suggests. 

\begin{eg}\label{eg:PGL}
We have a commutative diagram 
\[
 \xymatrix{
 \Irr (\PGL_2(F)) 
 \ar[rr]^-{\mathrm{LL}_{\PGL_2}} \ar@{^{(}->}[d] & & 
 \Phi (\PGL_2) \ar@{^{(}->}[d] \\ 
 \Irr (\GL_2(F)) 
 \ar[rr]^-{\mathrm{LL}_{\GL_2}} & & 
 \Phi (\GL_2) 
 }
\]
by functoriality. 
On the other hand, there does not exist a map 
\[
 \mathrm{LL}_{\PGL_2,\ell} \colon 
 \Irr_{\ell} (\PGL_2(F)) \to 
 \Phi_{\ell} (\PGL_2) 
\]
which makes the commutative diagram 
\[
 \xymatrix{
 \Irr_{\ell} (\PGL_2(F)) 
 \ar[rr]^-{\mathrm{LL}_{\PGL_2,\ell}} \ar@{^{(}->}[d] & & 
 \Phi_{\ell} (\PGL_2) \ar@{^{(}->}[d] \\ 
 \Irr_{\ell} (\GL_2(F)) 
 \ar[rr]^-{\mathrm{LL}_{\GL_2,\ell}} & & 
 \Phi_{\ell} (\GL_2), 
 }
\]
because 
$(\det \circ p_{\wh{\GL_2}} \circ \varphi) (w)=\lvert w \rvert$ 
for $w \in W_F$ 
if $[\varphi] \in \Phi_{\ell} (\GL_2)$ is the image under 
$\mathrm{LL}_{\GL_2,\ell}$ of an element 
of $\Irr_{\ell} (\GL_2(F))$ 
coming from 
$\Irr_{\ell} (\PGL_2(F))$ by the construction of $\mathrm{LL}_{\GL_2,\ell}$ and \cite[10.1]{BorAutL}. 
\end{eg}

\subsection{$\ell$-adic C-parameter}
A C-group is constructed in 
\cite[Definition 5.38]{BuGeconjc} 
for a connected reductive group over a number field. 
We recall the construction here in our setting. 
Let $G^{\mathrm{ad}}$ be the adjoint quotient of $G$, 
and $G^{\mathrm{sc}}$ be the simply-connected cover of $G^{\mathrm{ad}}$. 
Let $\gamma \colon Z(G^{\mathrm{sc}}) \to \bG_{\mathrm{m}}$ 
be the restriction to $Z(G^{\mathrm{sc}})$ of 
the half sum of the positive roots of $G^{\mathrm{sc}}$, 
where we take a maximal torus $T^{\mathrm{sc}}$ and 
a Borel subgroup $B^{\mathrm{sc}}$ of $G^{\mathrm{sc}}_{\ol{F}}$ 
such that $T^{\mathrm{sc}} \subset B^{\mathrm{sc}}$ 
to define the positive roots, 
but $\gamma$ is independent of the choice. 
By pushing forward the exact sequence 
\begin{equation*}\label{eq:exGsc}
1 \to Z(G^{\mathrm{sc}}) \to G^{\mathrm{sc}} \to G^{\mathrm{ad}} \to 1 
\end{equation*}
by $\gamma$, 
we obtain an extension 
\begin{equation*}\label{eq:exG1}
1 \to \bG_{\mathrm{m}} \to G^1 \to G^{\mathrm{ad}} \to 1. 
\end{equation*}
By taking the pullback of this extension along 
the natural morphism $G \to G^{\mathrm{ad}}$, 
we obtain an extension 
\begin{equation*}\label{eq:exwtG}
1 \to \bG_{\mathrm{m}} \to \wt{G} \to G \to 1. 
\end{equation*}
We define the C-group 
${}^C G$ of $G$ as the L-group ${}^L \wt{G}$ 
of $\wt{G}$. 

The character 
\[
 G^{\mathrm{sc}} \times \bG_{\mathrm{m}} \to \bG_{\mathrm{m}};\ 
 (g,z) \mapsto z^2 
\]
induces a character $G^1 \to \bG_{\mathrm{m}}$, 
since $\gamma^2 =1$ by the construction of $\gamma$. 
It further induces a character 
$\wt{G} \to \bG_{\mathrm{m}}$ by taking composition with 
the natural morphism $\wt{G} \to G^1$. 
Hence we obtain a morphism 
\begin{equation}\label{wtGGGm}
 \wt{G} \to G \times \bG_{\mathrm{m}}. 
\end{equation}

We take a Borel subgroup 
$B \subset G_{\ol{F}}$ and 
a maximal torus $T \subset B$ over $\ol{F}$. 
Let $\rho_G$ denote the half sum of positive roots of $G$ 
with respect to $T$ and $B$. 
Then $2\rho_G$ defines a cocharacter 
$\delta_G \colon \bG_{\mathrm{m}} \to \wh{T}$. 
We put 
\[
 z_G =\delta_G (-1). 
\]
Then $z_G$ is central in $\wh{G}$ and independent of choices of 
$B$ and $T$ as in \cite[Proposition 5.39]{BuGeconjc}. 
By the independent of choices, we see that 
$z_G \in Z(\wh{G})^{\Gamma_F} \subset {}^L G$. 

Then the morphism \eqref{wtGGGm} induces the isomorphism 
\begin{equation}\label{eq:strCG}
 (\wh{G} \times \bG_{\mathrm{m}} 
 /\langle (z_G,-1) \rangle ) \rtimes W_F \simeq {}^C G 
\end{equation}
as in \cite[Proposition 5.39]{BuGeconjc}. 
We sometimes express a point of 
${}^C G$ as $[(g,z,w)]$ using the above isomorphism. 
We define 
$t_{\bG_{\mathrm{m}}} \colon {}^C G \to \bG_{\mathrm{m}}$ 
by $t_{\bG_{\mathrm{m}}}([(g,z,w)])=z^2$. 
We have an exact sequence 
\[
 1 \to {}^L G \to {}^C G 
 \stackrel{t_{\bG_{\mathrm{m}}}}{\lra} \bG_{\mathrm{m}} \to 1. 
\]

\begin{defn}\label{def:ellC}
An $\ell$-adic C-parameter for $G$ 
is an $\ell$-adic L-parameter $\varphi$ for $\wt{G}$ 
such that 
$(t_{\bG_{\mathrm{m}}} \circ \varphi) (w)=\lvert w \rvert$. 
\end{defn}

Let $\Phi_{\ell}^{\mathrm{C}}(G)$ denote the 
equivalence classes of $\ell$-adic C-parameters for $G$. 
We take $c \in \ol{\bQ}_{\ell}$ such that $c^2=q$. 
We define 
\[
 i_c \colon {}^L G \to {}^C G 
\] 
by 
$i_c (g,w)=[(g,c^{-d_F(w)},w)]$. 
For an $\ell$-adic L-parameter $\varphi$ for $G$, 
we put 
\[
 \varphi_c=i_c \circ \varphi. 
\]

\begin{lem}\label{lem:ellCell}
We have a bijection between 
the set of the $\ell$-adic L-parameters for $G$ 
and the set of the 
$\ell$-adic C-parameters for $G$ 
given by sending 
$\varphi$ to $\varphi_c$. 
Further, this induces a bijection 
$\Phi_{\ell}(G) \simeq \Phi_{\ell}^{\mathrm{C}}(G)$. 
\end{lem}
\begin{proof}
The first claim follows from the definitions. 
If two $\ell$-adic L-parameters for $\wt{G}$ 
are conjugate by an element of 
$\wh{\wt{G}}(\ol{\bQ}_{\ell})$, then they are 
conjugate by an element of $\wh{G}(\ol{\bQ}_{\ell})$, 
since $\wh{\wt{G}}(\ol{\bQ}_{\ell})$ is generated by 
$\wh{G}(\ol{\bQ}_{\ell})$ and 
$Z_{\wh{\wt{G}}(\ol{\bQ}_{\ell})}({}^C G(\ol{\bQ}_{\ell}))$. 
Hence the second claim follows from the first one. 
\end{proof}

We define a map   
\[
 \cC_{G,\ell}^{c} \colon 
 \Phi_{\ell}(G) \to \Phi_{\ell}^{\mathrm{C}}(G)
\]
by sending $[\varphi]$ to $[\varphi_c]$. 
Then $\cC_{G,\ell}^{c}$ is a bijection by Lemma \ref{lem:ellCell}. 
For an isomorphism $\iota \colon \bC \stackrel{\sim}{\to} \ol{\bQ}_{\ell}$, 
we define 
\[
 \mathrm{LL}_{G,\ell}^{\mathrm{C},\iota} \colon 
 \Irr_{\ell} (G(F)) \to \Phi_{\ell}^{\mathrm{C}}(G) 
\]
as $\cC_{G,\ell}^{\iota(q^{\frac{1}{2}})} \circ \mathrm{LL}_{G,\ell}^{\iota}$. 

\begin{conj}\label{conj:LLC}
The map $\mathrm{LL}_{G,\ell}^{\mathrm{C},\iota}$ is independent of 
a choice of $\iota \colon \bC \stackrel{\sim}{\to} \ol{\bQ}_{\ell}$. 
\end{conj}

\subsection{Tannakian $\ell$-adic L-parameter}
Assume that $C=\ol{\bQ}_{\ell}$. 
For a topological group $H$, 
let $\Rep_{\ol{\bQ}_{\ell}}(H)$ be 
the category of continuous finite dimensional 
representations of $H$ over $\ol{\bQ}_{\ell}$. 
For an algebraic group $H$ over a field, 
a character $\chi$ of $H$ and 
a cocharacter $\mu$ of $H$, 
we define $\langle \chi, \mu \rangle_H \in \bZ$ 
by 
\[
 (\chi \circ \mu)(z) =z^{\langle \chi, \mu \rangle_H}. 
\]
For a cocharacter 
$\mu \in X_*(T)$ of a torus $T$ over $\ol{F}$, 
let $\wh{\mu} \in X^*(\wh{T})$ denote the corresponding 
character of the dual torus $\wh{T}$. 

Let $\sM_G$ be the conjugacy classes of 
cocharacters $\bG_{\mathrm{m}} \to G_{\ol{F}}$.  
Let $[\mu] \in \sM_G$. We put 
\[
 d_G([\mu])= \langle 2\rho_G, \mu \rangle_T , 
\]
where we take a Borel subgroup 
$B \subset G_{\ol{F}}$, 
a maximal torus $T \subset B$ defined over $\ol{F}$ 
and a dominant representative $\mu \in X_*(T)$. 
Let 
$E_{[\mu]}$ be the field of definition of $[\mu]$. 
Let $r_{\wh{G},[\mu]}$ be the 
irreducible representation of $\wh{G}(\ol{\bQ}_{\ell})$ 
of highest weight $\wh{\mu}$ 
viewed as a dominant character of a maximal torus of $\wh{G}$.

We take $c \in \ol{\bQ}_{\ell}$ such that $c^2=q$. 
For an integer $m$, let 
\[
 \left( \frac{m}{2} \right)_c 
\]
denote the twist by the character 
$W_F \to \ol{\bQ}_{\ell}$ sending $w$ to 
$c^{-md_F(w)}$. 

Let $\Rep_{\ol{\bQ}_{\ell}}^{\mathrm{alg}} ({}^L G)$ denote the 
category of continuous finite dimensional representations of 
${}^L G (\ol{\bQ}_{\ell})$ over $\ol{\bQ}_{\ell}$ 
whose restrictions to $\wh{G}(\ol{\bQ}_{\ell})$ 
are algebraic. 
Let 
$r \colon {}^L G (\ol{\bQ}_{\ell}) \to \Aut (V)$ 
be an object in $\Rep_{\ol{\bQ}_{\ell}}^{\mathrm{alg}} ({}^L G)$. 
Then we have a decomposition 
\[
 V=\bigoplus_{[\mu] \in \sM_G} V_{[\mu]} 
\]
as representations of $\wh{G}(\ol{\bQ}_{\ell})$ 
where $V_{[\mu]}$ is the 
$r_{\wh{G},[\mu]}$-typic part of $V$. 
For an $\ell$-adic L-parameter $\varphi$, 
we define $(r \circ \varphi)_c \colon W_F \to \Aut (V)$ 
by 
\[
 V=\bigoplus_{[\mu] \in \cM_G} V_{[\mu]} \left( \frac{d_G([\mu])}{2}\right)_c, 
\]
which means that we twist 
$r \circ \varphi \colon W_F \to \Aut (V)$ 
by 
\[
 \left( \frac{d_G([\mu])}{2}\right)_c 
\] 
on each direct summand $V_{[\mu]}$. 
This is well-defined, because $d_G(w[\mu])=d_G([\mu])$ for $w \in W_F$. 

For an $\ell$-adic L-parameter $\varphi$ for $G$, 
we define a tensor functor 
\[
 \cF_{\varphi,c} \colon \Rep_{\ol{\bQ}_{\ell}}^{\mathrm{alg}} ({}^L G) 
 \to \Rep_{\ol{\bQ}_{\ell}} (W_F) 
\]
by 
\[
 \cF_{\varphi,c} (r) =(r \circ \varphi)_c . 
\]

\begin{defn}\label{def:Tell}
A Tannakian $\ell$-adic L-parameter for $G$ is a functor 
\[
 \cF \colon \Rep_{\ol{\bQ}_{\ell}}^{\mathrm{alg}} ({}^L G) 
 \to \Rep_{\ol{\bQ}_{\ell}} (W_F) 
\]
which is equal to $\cF_{\varphi,c}$ for 
an $\ell$-adic L-parameter $\varphi$ for $G$. 
We say that two Tannakian $\ell$-adic L-parameters 
$\cF$ and $\cF'$ 
are equivalent if 
there is $g \in \wh{G}(\ol{\bQ}_{\ell})$ such that, 
for all $r \in \Rep_{\ol{\bQ}_{\ell}}^{\mathrm{alg}} ({}^L G)$, we have 
$\cF(r)=\cF'(r)^{r(g)}$. 
\end{defn}

\begin{lem}\label{lem:TLind}
The set of the Tannakian $\ell$-adic L-parameters for $G$ 
is independent of a choice of 
$c \in \ol{\bQ}_{\ell}$ such that $c^2=q$.
\end{lem}
\begin{proof}
We have 
\begin{equation}\label{eq:mudG}
 \wh{\mu} (z_G) =(-1)^{\langle \wh{\mu} ,\delta_G \rangle_{\wh{T}}}
 =(-1)^{d_G([\mu])}
\end{equation}
for $\mu \in X_*(T)$. 
Let 
\[
 \omega_{z_G} \colon W_F \to Z(\wh{G})^{\Gamma_F}(\ol{\bQ}_{\ell}) 
 \hookrightarrow {}^L G (\ol{\bQ}_{\ell}) 
\]
be the character 
sending $w$ to $z_G^{d_F(w)}$. 
By \eqref{eq:mudG}, we have 
\begin{equation}\label{eq:-crel}
 (r \circ \varphi)_{-c} = (r \circ ( \omega_{z_G} \varphi ))_c 
\end{equation}
for an $\ell$-adic L-parameter $\varphi$ 
and $r \in \Rep_{\ol{\bQ}_{\ell}}^{\mathrm{alg}} ({}^L G)$. 
Since $\omega_{z_G} \varphi$ is also an $\ell$-adic L-parameter, 
the claim follows. 
\end{proof}

Let $\Phi_{\ell}^{\mathrm{T}}(G)$ be 
the set of equivalence classes of 
Tannakian $\ell$-adic L-parameters for $G$. 
This set is independent of 
a choice of $c$ by Lemma \ref{lem:TLind}.

\begin{lem}\label{lem:ellTell}
We have a bijection between 
the set of the $\ell$-adic L-parameters for $G$ 
and the set of the 
Tannakian $\ell$-adic L-parameters for $G$ 
given by sending 
$\varphi$ to $F_{\varphi,c}$. 
Further, this induces a bijection 
$\Phi_{\ell}(G) \simeq \Phi_{\ell}^{\mathrm{T}}(G)$. 
\end{lem}
\begin{proof}
We show the first claim. 
The map is surjective by Definition \ref{def:Tell}. 
Assume that $\varphi$ and $\varphi'$ 
are different $\ell$-adic L-parameters for $G$ and 
$\cF_{\varphi,c}=\cF_{\varphi',c}$. 
We take $w \in W_F$ such that 
$\varphi (w) \neq \varphi'(w)$. 
Further, we take a finite Galois extension $F'$ of $F$ 
such that 
$G$ splits over $F'$ and 
the images of $\varphi (w)$ and $\varphi'(w)$ 
in $\wh{G}(\ol{\bQ}_{\ell}) \rtimes \Gal (F'/F)$ 
are different. 
By considering a representation of 
${}^L G (\ol{\bQ}_{\ell})$ 
which factors through a faithful algebraic representation of 
$\wh{G}(\ol{\bQ}_{\ell}) \rtimes \Gal (F'/F)$, 
we have a contradiction to $\cF_{\varphi,c}=\cF_{\varphi',c}$. 
Hence the map is injective. 

The second claim follows from the first one. 
\end{proof}

For a Tannakian $\ell$-adic L-parameter $\cF$ for $G$, 
we take 
an $\ell$-adic L-parameter $\varphi$ for $G$ 
such that $\cF =\cF_{\varphi,c}$. 
Then the centralizer $S_{\varphi}=Z_{\wh{G}(\ol{\bQ}_{\ell})}(\Image \varphi)$ 
is independent of a choice of $c$ by \eqref{eq:-crel}. 
We write $S_{\cF}$ for $S_{\varphi}$. 
Then $\cF$ naturally factors through 
\[
 \cF_S \colon \Rep_{\ol{\bQ}_{\ell}}^{\mathrm{alg}} ({}^L G) 
 \to \Rep_{\ol{\bQ}_{\ell}} (S_{\cF} \times W_F). 
\]

For a finite separable extension $F'$ of $F$, 
we define the restriction 
\[
 \cF|_{F'} \colon 
 \Rep_{\ol{\bQ}_{\ell}}^{\mathrm{alg}} ({}^L G_{F'}) 
 \to \Rep_{\ol{\bQ}_{\ell}} (W_{F'}) 
\]
of $\cF$ to $F'$ 
by the usual restriction to $W_{F'}$ of an 
$\ell$-adic L-parameter for $G$ and 
bijections given by Lemma \ref{lem:ellTell}. 
Let 
\[
 \cF_S|_{F'} \colon \Rep_{\ol{\bQ}_{\ell}}^{\mathrm{alg}} ({}^L G_{F'}) 
 \to \Rep_{\ol{\bQ}_{\ell}} (S_{\cF} \times W_{F'}) 
\]
be the composition of 
\[
 (\cF|_{F'})_S \colon \Rep_{\ol{\bQ}_{\ell}}^{\mathrm{alg}} ({}^L G_{F'}) 
 \to \Rep_{\ol{\bQ}_{\ell}} (S_{\cF|_{F'}} \times W_{F'}) 
\]
and the natural functor 
\[
 \Rep_{\ol{\bQ}_{\ell}} (S_{\cF|_{F'}} \times W_{F'})
 \to \Rep_{\ol{\bQ}_{\ell}} (S_{\cF} \times W_{F'}) 
\]
induced by the restriction with respect to 
$S_{\cF} \subset S_{\cF|_{F'}}$. 

We define a map 
\[
 \cT_{G,\ell}^{c} \colon 
 \Phi_{\ell}(G) \to \Phi_{\ell}^{\mathrm{T}}(G)
\]
by sending 
$[\varphi]$ to $[\cF_{\varphi,c}]$. 
Then $\cT_{G,\ell}^{c}$ is a bijection 
by Lemma \ref{lem:ellTell}. 
For an isomorphism $\iota \colon \bC \stackrel{\sim}{\to} \ol{\bQ}_{\ell}$, 
we define 
\[
 \mathrm{LL}_{G,\ell}^{\mathrm{T},\iota} \colon 
 \Irr_{\ell} (G(F)) \to \Phi_{\ell}^{\mathrm{T}}(G) 
\]
as $\cT_{G,\ell}^{\iota (q^{\frac{1}{2}})} \circ \mathrm{LL}_{G,\ell}^{\iota}$. 

\begin{conj}\label{conj:LLT}
The map $\mathrm{LL}_{G,\ell}^{\mathrm{T},\iota}$ is independent of 
a choice of $\iota \colon \bC \stackrel{\sim}{\to} \ol{\bQ}_{\ell}$. 
\end{conj}

\begin{rem}\label{rem:KotConj}
Conjecture \ref{conj:LLT} is motivated by 
the Kottwitz conjecture for local Shimura varieties in 
\cite[Conjecture 7.4]{RVlocSh}. 
Let $(G,[b],[\mu])$ be a local Shimura datum as \cite[Definition 5.1]{RVlocSh}. Let $\cM_{G,[b],[\mu],K}$ be the local Shimura variety over the reflex field $E_{[\mu]}$ 
attached to $(G,[b],[\mu])$ and $K \subset G(F)$, which is constructed in \cite[24.1]{ScWeBLp}. 
Let $J$ denote the $\sigma$-centralizer of $b$. 
Let $[\rho] \in \Irr_{\ell} (J)$. 
We put 
\[
H^{\bullet} ((G,[b],[\mu]))[\rho]= (-1)^{d_G([\mu])}
\sum_{i,j \geq 0} (-1)^{i+j} H^{i,j}((G,[b],[\mu]))[\rho]  
\]
where  
\[
H^{i,j}((G,[b],[\mu]))[\rho] = \varinjlim_{K} 
\Ext_{J(F)}^j (H^i_{\mathrm{c}} (\cM_{G,[b],[\mu],K,\widehat{\ol{E}}},\ol{\bQ}_{\ell}),\rho) . 
\]
We put 
\[
 [\cF^{\iota}] 
 =\mathrm{LL}_{J,\ell}^{\mathrm{T},\iota}([\rho]) . 
\]
For $[\pi] \in \Irr_{\ell} (G)$ such that 
$\mathrm{LL}_{G,\ell}^{\mathrm{T},\iota}([\pi])=[\cF^{\iota}]$, 
let 
$\delta_{\pi,\rho}^{\iota}$ be the representation 
of $S_{\cF^{\iota}}$ over $\ol{\bQ}_{\ell}$ 
determined by $\iota$ and 
$\check{\tau}_{\pi'} \otimes \tau_{\rho'}$ constructed 
in \cite[p.~312]{RVlocSh} 
(\cf \cite[2.3]{HKWKotloc} for a construction in a more general case), 
where 
$\pi'=\pi \otimes_{\ol{\bQ}_{\ell},\iota^{-1}} \bC$ and 
$\rho'=\rho \otimes_{\ol{\bQ}_{\ell},\iota^{-1}} \bC$. 
Let $r_{[\mu]}$ be an extension of 
$r_{\wh{G},[\mu]}$ to ${}^L G_{E_{[\mu]}}$ 
constructed by \cite[(1.1.3), (2.1.2)]{KotShtw} 
using \cite[2.4 Remark (3)]{BorAutL}. 
Then the Kottwitz conjecture says that 
\[
 H^{\bullet} ((G,[b],[\mu]))[\rho]=\sum_{[\pi]} \pi \boxtimes 
  \Hom_{S_{\cF^{\iota}}} (\delta_{\pi,\rho}^{\iota}, 
 \cF_S^{\iota}|_{E_{[\mu]}} 
 (r_{[\mu]})) ,  
\]
where $[\pi]$ runs $[\pi] \in \Irr_{\ell} (G)$ such that 
$\mathrm{LL}_{G,\ell}^{\mathrm{T},\iota}([\pi])=[\cF^{\iota}]$. 
If the Kottwitz conjecture is true, 
then the isomorphism class of the $W_{E_{[\mu]}}$-representation 
\[
 \Hom_{S_{\cF^{\iota}}} (\delta_{\pi,\rho}^{\iota}, 
 \cF_S^{\iota}|_{E_{[\mu]}} 
 (r_{[\mu]}))
\]
is independent of $\iota$, since the $\ell$-adic cohomology of 
local Shimura varieties and their group actions are 
independent of $\iota$. 
If Conjecture \ref{conj:LLT} is true, 
$\cF^{\iota}_S$ and $S_{\cF^{\iota}}$ are independent of $\iota$. 
By the observation above, we conjecture also that 
$\delta_{\pi,\rho}^{\iota}$ is independent of $\iota$. 
\end{rem}

\subsection{Comparison}

\begin{thm}\label{thm:CTbij}
There is a canonical bijection 
\[
 \mathcal{CT}_{G,\ell} \colon \Phi_{\ell}^{\mathrm{C}}(G) \to 
 \Phi_{\ell}^{\mathrm{T}}(G) 
\]
such that for any square root $c$ of $q$ in $\ol{\bQ}_{\ell}$
the diagram 
\begin{equation*}\label{eq:CTcomm}
  \xymatrix{
		\Phi_{\ell} (G) 
		\ar[rr]^-{\cC_{G,\ell}^c} \ar[rrd]_-{\cT_{G,\ell}^c} & & 
		\Phi_{\ell}^{\mathrm{C}}(G) \ar[d]^{\mathcal{CT}_{G,\ell}} \\ 
		& & 
		\Phi_{\ell}^{\mathrm{T}}(G) 
	}	
\end{equation*}
is commutative. 
\end{thm}
\begin{proof}
Let $\wt{\varphi}$ be an $\ell$-adic C-parameter for $G$. 
We construct an $\ell$-adic T-parameter $\cF_{\wt{\varphi}}$ for $G$. 
Let 
$r \colon {}^L G (\ol{\bQ}_{\ell}) \to \Aut (V)$ 
be an object in $\Rep_{\ol{\bQ}_{\ell}}^{\mathrm{alg}} ({}^L G)$. 
Then we have a decomposition 
\[
V=\bigoplus_{[\mu] \in \sM_G} V_{[\mu]} 
\]
as representations of $\wh{G}(\ol{\bQ}_{\ell})$ 
where $V_{[\mu]}$ is the 
$r_{\wh{G},[\mu]}$-typic part of $V$. 
We extend $r$ to 
$\wt{r} \in \Rep_{\ol{\bQ}_{\ell}}^{\mathrm{alg}} ({}^C G)$ 
by letting $\bG_{\mathrm{m}}$ act on 
$V_{[\mu]}$ by $z \mapsto z^{d_G([\mu])}$ 
using \eqref{eq:strCG}, 
where the extension is well-defined by \eqref{eq:mudG}. 
Then we define $\cF_{\wt{\varphi}}$ by 
\begin{equation*}
	\cF_{\wt{\varphi}}(r)=\wt{r} \circ \wt{\varphi}. 
\end{equation*}
If two $\ell$-adic C-parameters for $G$ are conjugate by an element of 
$\wh{\wt{G}}(\ol{\bQ}_{\ell})$, then they are conjugate by an element of 
$\wh{G}(\ol{\bQ}_{\ell})$. 
Hence, 
\[
 \mathcal{CT}_{G,\ell} ([\wt{\varphi}]) = [\cF_{\wt{\varphi}}] 
\]
is well-defined. 
We have the commutative diagram in the claim 
by 
	\begin{equation}\label{eq:wtrF}
	\wt{r} \circ \varphi_c=\wt{r} \circ i_c \circ \varphi=\cF_{\varphi,c}(r). 
\end{equation}
Then $\mathcal{CT}_{G,\ell}$ is a bijection because $\cC_{G,\ell}^c$ and $\cT_{G,\ell}^c$ are bijections by Lemma \ref{lem:ellCell} and Lemma \ref{lem:ellTell}. 
\end{proof}

\begin{cor}\label{cor:CTeq}
Conjecture \ref{conj:LLC} and 
Conjecture \ref{conj:LLT} are equivalent. 
\end{cor}
\begin{proof}
We have 
\[
\mathcal{CT}_{G,\ell} \circ \mathrm{LL}_{G,\ell}^{\mathrm{C},\iota}=\mathrm{LL}_{G,\ell}^{\mathrm{T},\iota} 
\]
for any $\iota \colon \bC \stackrel{\sim}{\to} \ol{\bQ}_{\ell}$ by 
Theorem \ref{thm:CTbij}. 
Since $\mathcal{CT}_{G,\ell}$ is a canonical bijection independent of $\iota$, the claim follows. 
\end{proof}

\begin{cor}\label{cor:conjGL}
Conjecture \ref{conj:LLC} and 
Conjecture \ref{conj:LLT} 
are true for $\GL_n$. 
\end{cor}
\begin{proof}
Recall that the local Langlands correspondence for $\GL_n$ 
is known by \cite{LRSellL} and \cite{HTsimSh}. 
By Corollary \ref{cor:CTeq}, it suffices to check Conjecture 
\ref{conj:LLC}. 
Assume that 
$\mathrm{LL}_{\GL_n,\ell}^{\mathrm{C},\iota}([\pi])
=[\varphi^{\iota}]$. 
Note that $z_{\GL_n}=(-1)^{n-1}$, because 
\[
 \delta_{\GL_n} \colon \bG_{\mathrm{m}} \to 
 \wh{T_n} \subset \wh{\GL_n} ; \ 
 z \mapsto 
 \begin{pmatrix}
      z^{n-1} & 0 & \ldots & 0 \\
      0 & z^{n-3} & \ldots & 0 \\
      \vdots & \vdots & \ddots & \vdots \\
      0 & 0 & \ldots & z^{1-n}
 \end{pmatrix}, 
\]
where we define $\delta_{\GL_n}$ using 
the diagonal maximal torus $T_n$ of $\GL_n$ 
and the Borel subgroup of the 
upper triangular matrices in $\GL_n$. 
Let 
\[
 \wt{r} \colon {}^C \GL_n \to \wh{\GL_n} ;\ 
 [(g,z,w)] \mapsto gz^{n-1}. 
\]
We show that 
$[ \wt{r} \circ \varphi^{\iota} ]$ 
is independent of $\iota$ in a similar way as the proof of \cite[35.1 Theorem]{BHLLCGL2}. 
For $[\rho] \in \Irr (\GL_n)$, 
we write $\mathrm{LL}_{\GL_n}^{\mathrm{WD}}([\rho])$ for the image of 
$\mathrm{LL}_{\GL_n}([\rho])$ under $\Phi (\GL_n) \simeq \Phi_{\bC}^{\mathrm{WD}}(\GL_n)$, 
and define $\tau_n ([\rho])$ as the twist of 
$\mathrm{LL}^{\mathrm{WD}}_{\GL_n}([\rho])$ by $(a,w) \mapsto q^{-(n-1)d_F(w)/2}$. 
Then $\tau_n ([\pi \otimes_{\ol{\bQ}_{\ell},\iota^{-1}} \bC])$ 
and $[ \wt{r} \circ \varphi^{\iota} ]$ 
corresponds under the identification by 
\[
 \Phi_{\bC}^{\mathrm{WD}} (\GL_n) 
\stackrel{\Phi_{\iota}^{\mathrm{WD}} (\GL_n)}{\simeq} \Phi_{\ol{\bQ}_{\ell}}^{\mathrm{WD}} (\GL_n) 
\simeq \Phi_{\ell} (\GL_n). 
\]
Therefore it suffices to show that 
$[\rho] \mapsto \tau_n ([\rho])$ is compatible with twists by 
$\Aut (\bC)$. 
By \cite[1.8 Corollaire]{HenUnecar}, 
$\{ \tau_n \}_{n \geq 1}$ is characterized by 
the compatibility of $\tau_1$ with the local class field theory 
and the equalities 
\begin{align*}
	L\left( [\rho] \times [\rho'],s+\frac{n+n'}{2}-1\right) &= L(\tau_n ([\rho]) \otimes \tau_{n'} ([\rho']),s),\\ 
	\varepsilon \left([\rho] \times [\rho'],s+\frac{n+n'}{2}-1,\psi \right)&= \varepsilon(\tau_n ([\rho]) \otimes \tau_{n'} ([\rho']),s,\psi)   
\end{align*}
for $n' <n$, $[\rho] \in \Irr (\GL_{n})$ and $[\rho'] \in \Irr (\GL_{n'})$, where $\psi$ is a non-trivial character of $F$. 
The compatibility with twists by $\Aut (\bC)$ follows from this characterization, \cite[3.2 Theorem]{BuHeDHII} and 
\cite[35.3]{BHLLCGL2}. 

Let $\iota' \colon \bC \stackrel{\sim}{\to} \ol{\bQ}_{\ell}$ 
be another choice. 
Taking a conjugation of $\varphi^{\iota'}$ by an element of 
$\wh{\GL_n}(\ol{\bQ}_{\ell})$, 
we may assume that 
$\wt{r} \circ \varphi^{\iota}=\wt{r} \circ \varphi^{\iota'}$. 
Hence there is a map $\chi \colon W_F \to \ol{\bQ}_{\ell}^{\times}$ such that 
\[
 \varphi^{\iota} (w) =\varphi^{\iota'} (w) 
 [(\chi(w)^{1-n},\chi(w),1) ] 
\]
for $w \in W_F$. 
By Definition \ref{def:ellC}, 
we have 
$t_{\bG_{\mathrm{m}}} \circ \varphi^{\iota}
 =t_{\bG_{\mathrm{m}}} \circ \varphi^{\iota'}$. 
Hence we have $\chi (w)^2=1$ for $w \in W_F$. 
This implies 
that $\varphi^{\iota} =\varphi^{\iota'}$ by 
$z_{\GL_n}=(-1)^{n-1}$ and \eqref{eq:strCG}. 
Hence Conjecture \ref{conj:LLC} is true. 
\end{proof}

\begin{rem}
We can also show Corollary \ref{cor:conjGL} using the geometric realization of the local Langlands correspondence for $\GL_n$ in 
the $\ell$-adic etale cohomology of the Lubin--Tate spaces after the reduction to the supercuspidal case in the same spirit as Remark \ref{rem:KotConj} (\cf \cite[Lemma VII.1.6]{HTsimSh}). 
Here we gave a proof by the characterization without appealing to such a geometric realization. 
\end{rem}

\begin{cor}\label{cor:conjPGL}
	Conjecture \ref{conj:LLC} and 
	Conjecture \ref{conj:LLT} 
	are true for $\PGL_n$. 
\end{cor}
\begin{proof}
By Corollary \ref{cor:CTeq}, it suffices to check Conjecture 
\ref{conj:LLC}. This follows from Corollary \ref{cor:conjGL} and 
the commutative diagram 
\[
\xymatrix{
	\Irr_{\ell} (\PGL_n(F)) 
	\ar[rr]^-{\mathrm{LL}_{\PGL_n,\ell}^{\iota}} \ar@{^{(}->}[d] & & 
	\Phi_{\ell} (\PGL_n) \ar@{^{(}->}[d] \ar[rr]^-{\cC_{\PGL_n,\ell}^{\iota(q^{\frac{1}{2}})}} & &  \Phi_{\ell}^{\mathrm{C}} (\PGL_n)  \ar@{^{(}->}[d]\\ 
	\Irr_{\ell} (\GL_n(F)) 
	\ar[rr]^-{\mathrm{LL}_{\GL_n,\ell}^{\iota}} & & 
	\Phi_{\ell} (\GL_n) \ar[rr]^-{\cC_{\GL_n,\ell}^{\iota(q^{\frac{1}{2}})}} & &  \Phi_{\ell}^{\mathrm{C}} (\GL_n) 
}
\]
since the vertical injections are independent of $\iota$. 
\end{proof}


\noindent
Naoki Imai\\
Graduate School of Mathematical Sciences, The University of Tokyo, 
3-8-1 Komaba, Meguro-ku, Tokyo, 153-8914, Japan \\
naoki@ms.u-tokyo.ac.jp 

\end{document}